\documentclass[reqno,12pt]{amsart}

\usepackage[utf8]{inputenc}

\usepackage{enumerate}
\usepackage[margin=1in,marginparwidth=0.7in]{geometry}

\usepackage{ifpdf}
\usepackage{amsmath}
\usepackage{amsfonts}
\usepackage{amssymb}
\usepackage{amsthm}
\usepackage[ocgcolorlinks,hyperfootnotes=false,colorlinks=true,citecolor=blue,linkcolor=blue,urlcolor=blue]{hyperref}
\usepackage{setspace}
\usepackage{amsrefs}
\usepackage{nicefrac}
\usepackage{graphicx}
\usepackage{color}
\usepackage{mathtools}
\newcommand{\ignore}[1]{}

\renewcommand{\Re}{\operatorname{Re}}
\renewcommand{\Im}{\operatorname{Im}}

\newcommand{\sabs}[1]{\lvert {#1} \rvert}

\newcommand{\snorm}[1]{\lVert {#1} \rVert}

\newcommand{\C}{{\mathbb{C}}}
\newcommand{\R}{{\mathbb{R}}}

\newcommand{\D}{{\mathbb{D}}}

\newcommand{\bT}{{\mathbb{T}}}

\newcommand{\sA}{{\mathcal{A}}}

\newcommand{\sD}{{\mathcal{D}}}

\newcommand{\sO}{{\mathcal{O}}}

\newcommand{\sU}{{\mathcal{U}}}

\newtheorem{thm}{Theorem}[section]

\newtheorem{cor}[thm]{Corollary}

\newtheorem{lemma}[thm]{Lemma}

\newcommand{\innerthmname}{}% initialize

\theoremstyle{definition}
\newtheorem{defn}[thm]{Definition}
\newtheorem{example}[thm]{Example}

\theoremstyle{remark}

\ifpdf
\hypersetup{
  pdftitle={Cartan uniqueness theorem on nonopen sets},
  pdfauthor={Jiri Lebl, Alan Noell, Sivaguru Ravisankar},
  pdfsubject={Several Complex Variables},
  pdfkeywords={Cartan's uniqueness theorem, real-analytic subvariety, CR function, disc hull},
}
\fi

\author{Ji\v{r}\'{\i} Lebl}
\thanks{The first author was in part supported by Simons Foundation collaboration grant 710294.}
\address{Department of Mathematics, Oklahoma State University,
Stillwater, OK 74078, USA}
\email{lebl@okstate.edu}

\author{Alan Noell}
\address{Department of Mathematics, Oklahoma State University,
Stillwater, OK 74078, USA}
\email{alan.noell@okstate.edu}

\author{Sivaguru Ravisankar}
\thanks{The last author was in part supported by Mathematical Research Impact Centric Support (MATRICS) grant MTR/2022/000865 from the Science and Engineering Research Board (SERB), Government of India.}
\address{Tata Institute of Fundamental Research, Centre for Applicable Mathematics, Bengaluru 560065, India}
\email{sivaguru@tifrbng.res.in}

\date{\today}
\title{Cartan uniqueness theorem on nonopen sets}

\keywords{Cartan uniqueness theorem, real-analytic subvariety, CR function, disc hull}
\subjclass[2020]{32H02 (Primary),  32V40 32B20 (Secondary)}

\begin{document}

\begin{abstract}
Cartan's uniqueness theorem does not hold in general for CR mappings,
but it does hold under certain conditions guaranteeing extendibility of
CR functions to a fixed neighborhood.  These conditions can be defined naturally for
a wide class of sets such as local real-analytic subvarieties or
subanalytic sets, not just submanifolds.  Suppose that $V$ is a locally connected and locally
closed subset of $\C^n$ such that the hull constructed by contracting analytic discs
close to arbitrarily small neighborhoods of a point always contains the
point in the interior. Then restrictions of holomorphic functions uniquely extend
to a fixed neighborhood of the point.  Using this extension, we obtain
a version of Cartan's uniqueness theorem for such sets.
When $V$ is a real-analytic subvariety, we can generalize the concept of
infinitesimal CR automorphism and also prove an analogue of the theorem.
As an application of these two results we show that, for circular subvarieties
satisfying the condition, the only
automorphisms, CR or infinitesimal, are linear.
\end{abstract}

\maketitle

\textbf{Note:} The definition of the contracting disc hull has been changed
with respect to the published version.  An extra hypothesis is necessary.
The main issue is that Lemma 2.1 needs to be replaced.
See section~\ref{section:exampleerratum} for more details of what changed.
An unrelated minor issue was that Lemma 3.1 is missing a hypothesis in the
published version.

%\enlargethispage{\baselineskip}

%%%%%%%%%%%%%%%%%%%%%%%%%%%%%%%%%%%%%%%%%%%%%%%%%%%%%%%%%%%%%%%%%%%%%%%%%%%%

\section{Introduction} \label{section:intro}

It is a classical theorem of Cartan that, given a bounded domain $U \subset
\C^n$, a point $p \in U$, and a holomorphic mapping $f \colon U \to U$ with
$f(p)=p$ and $Df(p) = I$, then $f$ is the identity map.
A similar theorem
does not in general hold if $U$ is replaced by a real submanifold or
subvariety and $f$ is replaced by a CR mapping (see Example~\ref{example:failure}).  Since the classical
theorem is useful in computing the automorphisms of $U$, we wished to
investigate when a CR analogue of this theorem holds. We also
wanted to investigate local infinitesimal CR automorphisms, so we wished to
obtain a local CR analogue.

The types of sets we consider are fairly general,
such as real-analytic subvarieties, semianalytic sets, or subanalytic sets.
For our main result, it in fact suffices that the set is locally connected and
locally closed.
%Consider a possibly singular real-analytic local variety $V \subset
%\C^n$.  By local variety we mean that it is a closed subvariety of some
%open subset, but it is not necessarily a closed subset of $\C^n$ itself.
Let $V \subset \C^n$ be such a set.
As $V$ is not a CR submanifold, we consider $\sO(V)$,
the class of functions or mappings
that are restrictions to $V$ of functions or mappings holomorphic
in some neighborhood of $V$.  When $V$ is a real-analytic CR submanifold, this class
is the same as the set of real-analytic CR functions or mappings.  A key point in
this study is that the neighborhood to which such a function extends need
not be a fixed neighborhood of $V$ unless some extra geometric
condition is satisfied by $V$.

As we are interested in extension of CR functions into some fixed
neighborhood, the geometric condition we have in mind is a kind of hull.
Given a set $K \subset \C^n$ and $p \in K$, let $\widehat{K}_{CD,p}$ be
the \emph{contracting disc hull at $p$}, defined as follows.
First, by an analytic disc we mean a continuous map $\varphi \colon
\overline{\D} \to \C^n$ such that $\varphi$ is holomorphic in $\D$. (Here $\D$ is the unit disc in $\C$.)
By a continuous family of analytic discs, we mean a continuous map that depends
on an additional real variable.  We are more interested in the image of $\varphi$
than in the map itself, so we may refer to the image as the disc
rather than the map.

We define
$\widehat{K}_{CD,p}$ to be the set of points $z \in \C^n$ such that for
every $\epsilon > 0$, there is
a continuous family of analytic discs
$\varphi_{t} \colon \overline{\D} \to \C^n$ ($t \in [0, 1]$)
such that the boundary $\varphi_t(\partial \D)$ is within an
$\epsilon$-neighborhood of $K$ for all $t$,
$\varphi_1(0)=z$, $\varphi_0 \equiv p$,
and $\snorm{p-\varphi_t(0)}$ is
a strictly increasing function of $t$.

If we leave out the contracting condition, just assuming a single disc
through a point, we obtain the \emph{disc hull}, which we denote as $\widehat{K}_{D}$.
Such a hull is used in
constructing the polynomial hull of a set (see, e.g., \cites{Poletsky,Porten}).
However, if we leave out the contracting hypothesis, then functions in
$\sO(K)$ may not necessarily extend to such $\widehat{K}_{D}$.
A simple example is the unit
circle $\bT$ in $\C$, where $\widehat{\bT}_{D} = \overline{\D}$ but
$\widehat{\bT}_{CD,p} = \bT$.  Versions of the hull $\widehat{\bT}_{CD,p}$
can be used,
therefore, to construct the envelope of holomorphy (see also \cite{Joricke:09}).

\begin{defn}
A set $V \subset \C^n$ satisfies the \emph{contracting disc hull condition
at} $p\in V$ if, for every compact neighborhood $K \subset V$ of $p$
(neighborhood in the subspace topology), the set $\widehat{K}_{CD,p}$ contains
$p$ in its interior (in the topology of $\C^n$).
\end{defn}

Using the Kontinuit\"atssatz (see, e.g., \cite{Lebl:scv}*{Theorem 2.1.7 in
4th edition}), one can prove that
every function in $\sO(V)$ continues
to a holomorphic function on
some fixed neighborhood of any point of $K$ that is in the interior of
$\widehat{K}_{CD,p}$;
 by considering smaller and smaller $K$, we can show that the extension
is unique.  See Lemmas~\ref{lemma:extend} and \ref{lemma:unique}.
With such an extension, we can prove
a version of Cartan's uniqueness theorem (see Theorem~\ref{thm:funcs}):
If a connected and bounded real-analytic variety $V$ (or more generally a connected, bounded, locally connected, and locally closed set) satisfies the contracting disc hull condition at $p$, and if the
tangent cone of $V$  at $p$  is generic in the sense that it is not contained in a complex
hyperplane, then the identity map is the only self-mapping $f$ of $V$ in
$\sO(V)$ such that $f(p)=p$ and $Df(p)$ is the identity on the tangent cone.

An example of a set $V$ that satisfies this condition is the hypersurface
$\Im w = \sabs{z_1}^2 - \sabs{z_2}^2$ at $0$ (if we take the intersection with a ball containing $0$). This is a hypersurface
whose Levi-form has eigenvalues of both signs.  The condition gets
more complicated, however, when the dimension of $V$ is (much) smaller than a
hypersurface, and in some sense understanding a singular set of possibly
high codimension is our primary motivation.

The motivation for this work was to study the isotropy group at a point in
$V$.  The theorem can be thought of as a finite determination result for a certain
subclass of automorphisms.  The finite determination problem for an
automorphism is a topic with a very long history.  Chern and
Moser~\cite{ChernMoser:74} proved
that for a real-analytic Levi-nondegenerate hypersurface, the automorphisms are
determined by the 2-jet.
The finite determination question in various
dimensions and regularities has attracted much attention over the years
(see, e.g.,~\cites{Beloslapka:88,BER:98,ELZ:03,LamelMir:07,FLF:19,Tumanov:20} and the
references within).  See also the surveys by Zaitsev~\cite{Zaitsev:02} and
Baouendi--Ebenfelt--Rothschild~\cite{BER:00}.
A Cartan-like theorem is analogous to 1-jet determination.
However,
even automorphisms of the sphere are not determined by the 1-jet, so
further hypotheses on the set are necessary.
Cartan's theorem has also been generalized to almost complex manifolds by Lee~\cite{Lee:06}.

In addition to the local automorphisms, we also want to study the so-called
infinitesimal CR automorphisms.
For open sets the generalization to infinitesimal automorphisms was done by
Kaup and Upmeier~\cite{KaupUpmeier}.
We can define such automorphisms for
real-analytic subvarieties.  An infinitesimal CR automorphism is a vector
field $X$ such that $X = \Re Z$ for some holomorphic vector field $Z$ and
such that $X_q \in T_q V$ for all regular points $q$ of $V$.  Such a vector
field can be integrated into a family of biholomorphisms.

To obtain a version of Cartan's theorem, we need to generalize
two conditions on the mapping in the theorem: that the derivative is the identity at
a point, and that the mapping takes a bounded domain to itself (so that it
can be iterated).
First, the derivative being the identity simply says that the vector field is of
order 2 at $p$.
Second, requiring that the vector field can be iterated indefinitely can be replaced by
requiring that the flow of the vector field exist for all time.
With these conditions, and assuming as before that $V$
satisfies the contracting disc hull condition at $p$ and the tangent cone
at $p$ is generic,
we conclude that the infinitesimal automorphism $X$ is simply the zero vector field.
See Theorem~\ref{thm:vf}.

%%%%%%%%%%%%%%%%%%%%%%%%%%%%%%%%%%%%%%%%%%%%%%%%%%%%%%%%%%%%%%%%%%%%%%%%%%%%

\section{The theorem for mappings}

Our motivation is to prove a local version of Cartan's uniqueness
on real-analytic subvarieties, semianalytic sets, or even
subanalytic sets; however, the main feature of the type of sets
that we will require is that they are locally connected and locally closed.
Some of the results below hold in even more generality.

\begin{lemma} \label{lemma:extend}
Let $K \subset \C^n$ be a
compact and connected subset and $p \in K$.
Suppose that 
$\widehat{K}_{CD,p}$ has $p \in K$ in its interior,
and let $B \subset \widehat{K}_{CD,p}$ be a ball centered at $p$ such that $B \cap K$ is connected.
Let $f\in\sO(K)$.
Then there exists a holomorphic function $F \colon B \to \C$
such that $f$ and $F$ agree on $B \cap K$.
\end{lemma}

\begin{proof}
We assume $f$ is defined and holomorphic in some neighborhood $U$ of $K$.
Let $B = B_{\delta}(p)$ be as in the statement.
Clearly $f$ extends to some $B_{\delta'}(p)$ for $\delta' > 0$
as some ball lies completely in $U$.
Let $\delta'$ be the largest $\delta' \leq \delta$ such that $f$ has a
unique extension to $B_{\delta'}(p)$.  Suppose for a contradiction that $\delta' < \delta$.
Take $q \in \partial B_{\delta'}(p)$.
By hypothesis, there exists a path of
increasing distance from $p$ to $q$ given by a contracting family of
analytic discs whose boundaries stay inside $U$.
This path, except for the endpoint $q$, thus lies entirely in $B_{\delta'}(p)$.
By the Kontinuit\"atssatz, $f$ analytically continues along this path.
Thus there exists a small ball $\tilde{B}$ centered at $q$ such that
$f$ now uniquely extends to $B_{\delta'}(p) \cup \tilde{B}$.
As this construction can be done at every point in
$\partial B_{\delta'}(p)$, which is 
compact, we find that $f$ extends uniquely to some slightly larger $B_{\delta''}(p)$, providing a contradiction.
\end{proof}

\begin{lemma} \label{lemma:unique}
Let $V \subset \C^n$ be
a locally connected and locally closed set,
and suppose that $V$ satisfies the
contracting disc hull condition at $p \in V$.
Then there is a fixed neighborhood $\Omega \subset \C^n$ of $p$
such that for every function $f \in \sO(V)$ there is a unique
holomorphic $F \colon \Omega \to \C$ such that
$f$ and $F$ agree on $\Omega \cap V$.
\end{lemma}

\begin{proof}
Let $B \subset \C^n$ be a bounded neighborhood
of $p$ small enough such that
$K = V \cap \overline{B}$ is compact and connected.
We use the proof of Lemma~\ref{lemma:extend}
to find a neighborhood $\Omega$
of $p$ to which every function in $\sO(V)$ extends.
Hence, we get an extension of $f$ to $\Omega$
that agrees with $f$ on the
component of $K \cap \Omega$ through $p$.
We assume $\Omega \subset B$.
As $V$ (and hence $K$) is locally connected at $p$
we can also ensure that $\Omega$ is connected
and such that $K \cap \Omega$ is connected
and $\Omega \subset \widetilde{\Omega}$.
So for every $f \in \sO(V)$ there
is an $F$ holomorphic on $\Omega$
that agrees with $f$ on $K \cap \Omega = V \cap \Omega$.

Suppose $F$ and $G$ are two extensions of $f$ to $\Omega$.
Let $\overline{B_2} \subset \Omega$ be a closed ball
centered at $p$.
By the contracting disc hull condition, there is some
open neighborhood $\Omega'$ of the origin such that
every point $q \in \Omega'$ lies in some analytic disc
whose boundary can be made
arbitrarily close to $K_2 = V \cap \overline{B_2}$.
Suppose $B_3 \subset \Omega$ is a ball such that
$\overline{B_2} \subset B_3$. If the boundary of a disc is
close enough to $K_2$, that boundary lies in $B_3$,
and hence the entire disc
lies in $B_3 \subset \Omega$ (as $B_3$ is pseudoconvex).
As $F-G$ is zero on $K_2$, we find that it must
also be zero at $q$.
Thus $F=G$
on $\Omega'$ and hence on $\Omega$.
\end{proof}

We also need a version of the tangent space for a singular variety $V$.  One
natural possibility is the so-called tangent cone $C_p V$.
This cone is the set of
vectors $v$ that can be written as limits of sequences $r_j (q_j - p)$,
where $r_j > 0$ and $q_j \in V$.
We think of $C_p V$ as a subset of the vector space $T_p \C^n$.
We need that the cone is \emph{generic} in $T_p \C^n$ in the sense
that $C_p V$ is not contained in a proper complex
subspace, or in other words that the complex span of $C_p V$ is equal to
$T_p \C^n$.

The genericity of $C_p V$ implies that a complex linear mapping being the identity on the cone
means it is the identity for all vectors.  The exact statement that we want
is the following lemma, whose proof is immediate, as $Df(p)$ is a complex linear
mapping of $T_p \C^n$ to itself.

\begin{lemma} \label{lemma:identity}
Let $V \subset \C^n$ be a set, let
$p \in V$, and suppose that
the complex span of $C_p V$ is $T_p \C^n$.
Suppose that $f$ is a holomorphic mapping to $\C^n$
defined on a neighborhood of $V$ such that the derivative
$Df(p)$ restricted to $C_p V$ is the identity.
Then $Df(p)$ is the identity on $T_p \C^n$.
\end{lemma}

We can now prove a Cartan uniqueness theorem for real-analytic subvarieties, or actually  (as we said) for
locally connected and locally closed subsets of $\C^n$.

\begin{thm} \label{thm:funcs}
Let $V \subset \C^n$ be a connected, bounded, locally connected, and locally closed
set, and let $p \in V$.
Suppose that $V$ satisfies the contracting disc hull condition at $p$,
and suppose that the complex span of $C_p V$ is $T_p \C^n$.
Let $f \colon V \to V$ be a mapping in $\sO(V)$ such that
$f(p) = p$ and the (real) derivative $Df(p)$ restricted to $C_p V$ is the
identity.
Then $f(z) = z$ for all $z \in V$.
\end{thm}

\begin{proof}
By Lemma~\ref{lemma:unique}, there exists a neighborhood $\Omega$ of $p$
to which every function in $\sO(V)$ uniquely extends.
Let $\Delta$ be a polydisc centered at $p$ of radius $r$ such that
$\overline{\Delta} \subset \Omega$.
Let $f^{m} \colon V \to V$ be the $m$th iterate of $f$.  All the
iterates uniquely extend holomorphically to a neighborhood of
$\overline{\Delta}$,
and we denote by $F^m$ the extension of $f^m$.  For simplicity, write $F=F^1$.
Since $Df(p)$ restricted to $C_pV$ is the identity,  Lemma~\ref{lemma:identity} says that
the holomorphic derivative $DF$ at $p$ is the identity.  Similarly, the
derivative $DF^m$ at $p$ is the identity for all $m$.

Since any point of $\Delta$ lies on an analytic disc whose boundary is arbitrarily close
to $K$, the values of the extension of $f^{m}$ on $\Delta$ are bounded by the values
 of $f^{m}$ on $K$.  The functions $F^m$ all have  power series that
converge in $\Delta$.
If we consider a fixed $m$, we can consider $F^m$ to be the $m$th iterate of $F$ on
some small neighborhood of $p$: The $m$th iterate of $F$
restricted to $V$ has a unique holomorphic extension to some small
neighborhood of $\Delta$ where $F^{m}$ also is holomorphic.
Hence the two must have the same power series, so the power series of
$F^m$ behaves in the same way as that of an iterate of $F$.
The proof now follows in the same way as the standard proof
of Cartan's uniqueness theorem (see, e.g., \cite{Lebl:scv}*{\S 1.5}):
Without loss of generality, let $p=0$ and write
\begin{equation}
F(z) = z + \sum_{k=\ell}^\infty F_k(z) = z + F_\ell(z) +
\sum_{k=\ell+1}^\infty F_k(z) ,
\end{equation}
where $F_k$ homogeneous of degree $k$.
We can compute the same expression for $F^m$ by formally iterating the power series.
We find that the series for $F^m$ also starts at the $\ell$th step,
and the series for $F^m$ is
\begin{equation}
F^m(z) = z + m F_\ell(z) +
\sum_{k=\ell+1}^\infty F^m_k(z) .
\end{equation}
The $\ell$th-order terms are obtained from $F^m$
by an integral over some fixed torus $T$ in $\Delta$.  The values of $F^m$
are bounded by a fixed constant: The values of $F^m$
on $T$ are bounded by the values of $f^m$ on $K$, those values are in $V$, and $V$ is bounded.
Thus, letting $m$ go to infinity, we find that it must be that $F_\ell \equiv
0$, so $F(z)=z$ for all $z \in \Delta$. As $V$ is connected and $F$ is the unique holomorphic extension of $f$ to $\Omega$ (which we may take to be connected),
% hence $W$ can be taken to be connected and so
$f$ is also the identity.
\end{proof}

\begin{example}
Consider the submanifold $V$ of $\C^3$ given by $\Im z_3 = \sabs{z_1}^2-\sabs{z_2}^2$
and (to make it bounded) $\snorm{z} < 1$.
The submanifold $V$ satisfies the hypotheses of the theorem because an entire neighborhood of
the origin can always be filled by affine linear discs attached to $V$
(see, e.g., \cite{Lebl:scv}*{\S 3.4}).  By attached discs, we mean discs whose
boundary lies in $V$.  These discs shrink to the origin.
Furthermore, the tangent cone at the origin is just the tangent space there,
and it is a real hyperplane in $T_0 \C^3$. Thus, the cone is generic.  Any real-analytic CR
function (actually, any $C^1$-smooth CR function will do here) automatically
extends to a fixed neighborhood.
So the identity is the only
CR mapping that takes $V$ to itself, fixes the origin, and
has the identity as its derivative at the origin.
\end{example}

\begin{example} \label{example:failure}
The simplest example where the conclusion of the theorem does not hold is perhaps $V = \D
\times (-1,1) \subset \C^2$.
Here a mapping $f$ taking $V$ to itself is allowed to be an
arbitrary real-analytic mapping in the second (real) variable, so  $(z,t) \mapsto
(z,t-t^3)$ is an example.
\end{example}

\begin{example}
Consider $V \subset \C^6$ defined by
\begin{equation}
\Im z_1 = \sabs{z_2}^2-\sabs{z_3}^2 , \qquad
\Im z_4 = \sabs{z_5}^2-\sabs{z_6}^2 ,
\qquad \snorm{z} < 1 .
\end{equation}
Let $z = (z',z'') \in \C^3 \times \C^3$.  Any point $q = (q',q'')$ near the
origin lies on
an analytic disc attached to $V$.  To see this fact, consider $V$ as
$V' \times V''$.  We can find analytic discs $\varphi'$ and $\varphi''$
attached to $V'$ and $V''$ respectively such that
$\varphi'(0)=q'$ and
$\varphi''(0)=q''$.  Then the analytic disc we need is $\varphi' \times
\varphi''$, and such discs again shrink to the origin as required.
Furthermore, the tangent space at the origin is the subspace
$\{ \Im z_1 = \Im z_4 = 0 \}$, which is generic.  So $V$ satisfies the
hypotheses of the theorem, and $V$ is a CR submanifold (generic in fact).
\end{example}

\begin{example}
We now take the CR submanifold $V$ from the preceding example and consider its image under the
finite map $\Phi \colon \C^6 \to \C^6$ defined by
\begin{equation}
\Phi(z) =  \bigl(z_1 + i z_4, z_2, z_3, z_5, z_6, (z_1 + i z_4) z_2 \bigr) .
\end{equation}
Let $\varphi=\Phi|_V$. Then $\varphi(V)$ is a CR singular submanifold.
(Note that, although $\Phi$
is not a biholomorphism of any neighborhood of the origin, $\varphi$
is a diffeomorphism of $V$ onto $\varphi(V)$.)  Because a disc attached
to $V$ is mapped to a disc attached to $\varphi(V)$,
all the discs from the preceding example give
discs attached to $\varphi(V)$.  Also paths of increasing distance
stay paths of increasing distance.
For a compact set $K \subset \varphi(V)$, the hull
$\hat{K}_{CD,0}$ thus contains the pushforward
$\Phi\left(\widehat{\Phi^{-1}(K)}_{CD,0}\right)$.
As $\Phi$ is a finite holomorphic
map, this pushforward contains a whole neighborhood of the origin when
$\widehat{\Phi^{-1}(K)}_{CD,0}$ contains a whole neighborhood of the origin.
So $\varphi(V)$ also satisfies the contracting disc hull condition at the origin.

The tangent cone for $\varphi(V)$ at the origin is not generic because it is a complex hyperplane.
Thus the contracting disc hull condition does not imply the generic tangent
cone condition.
The conclusion of the theorem still holds, however:  Suppose that
$f$ is a map in $\sO(\varphi(V))$ taking $\varphi(V)$ to $\varphi(V)$. Assume that $f$ fixes the origin and its derivative is the identity
on $T_0 \varphi(V)$. We consider $\varphi^{-1} \circ f \circ \varphi$ on $V$.
It satisfies the derivative condition and is real-analytic. Since $\varphi$
extends locally to a biholomorphism on a relatively open dense subset of $V$,
 $\varphi^{-1} \circ f \circ \varphi$ is a CR map on such a subset of $V$, so it is a CR map on $V$.  Thus
$\varphi^{-1} \circ f \circ \varphi$ extends holomorphically to some neighborhood of $V$,
and hence it belongs to $\sO(V)$. We can now apply the uniqueness theorem on $V$. We conclude
that $\varphi^{-1} \circ f \circ \varphi$, and hence $f$, is the identity.
\end{example}

Using the previous example, one could prove a uniqueness result for any CR singular
CR image that is an image of a submanifold satisfying
the hypotheses of the theorem (see \cite{LNR:CRimages}).

\begin{example}
That the contracting disc hull condition does not imply the genericity
of the tangent cone is illustrated by the previous example.
Moreover, the contracting disc hull condition by itself is not enough
to imply the conclusion of the theorem.  Consider the semianalytic set
$V \subset \C^2$ given in coordinates $(z,w)$ by
\begin{equation}
\sabs{w} \leq \sabs{z}^2, \qquad
\sabs{z} \leq 1 .
\end{equation}
The tangent cone at the origin is simply $\{ w = 0 \}$, so the tangent cone is
not generic.  However, $V$ does satisfy the contracting disc hull condition
at the origin,
as clearly a neighborhood of the origin can be covered by affine analytic discs
(fix $w$ to be constant)
with boundary in $V$ that shrink to a disc lying in $\{ w = 0 \}$.

Consider the mapping $(z,w) \mapsto (z,w^2)$. It takes $V$ to $V$, it is
in $\sO(V)$, it fixes the origin, and
its derivative restricted to the tangent cone $C_0 V$ is the identity.
But the mapping is not the identity on $V$.
\end{example}

\section{Generic real-analytic submanifolds}

For a real-analytic CR submanifold $V$, any real-analytic CR function
belongs to $\sO(V)$ since it
extends holomorphically to a neighborhood of $V$.
Our condition on the tangent cone is simply that the CR manifold is in fact
a generic submanifold.
A real submanifold is \emph{generic} if at every point
the complex differentials of its
defining functions are $\C$ linearly independent.
In the real-analytic case, generic simply means
that the submanifold is CR and is not contained in any proper
complex submanifold.

For generic submanifolds,
the contracting disc hull condition follows from the existence of a single
analytic disc satisfying the right condition.
See \cite{BER:book} for the definition of defect.

\begin{lemma} \label{BER:lemma}
Let $V \subset \C^n$ be a generic real-analytic submanifold, and let $p \in V$
be a point such that there exists a small enough $C^{1,\alpha}$,
$0 < \alpha < 1$, analytic disc $\varphi \colon
\overline{\D} \to \C^n$ of defect $0$ attached to $V$ such that $\varphi(1) = p$ and
such that $\frac{d}{d t}|_{t=0} \varphi(e^{it}) \in T_p^c V$.
Then $V$ satisfies the contracting disc hull condition at $p$.
\end{lemma}

Here, the \emph{complex tangent space} $T_p^c V$ is the subspace of
$T_p V$ that is fixed by the
complex structure $J$, that is, $T_p^c V = T_p V \cap J(T_p V)$.
By ``small enough'' we mean that there exists $\epsilon > 0$ such that
$\varphi$ is within $\epsilon$ of the constant disc $p$,
where the distance used is the $C^{1,\alpha}$ norm.
The particular $\epsilon$
needed depends on $V$ and $p$ and is chosen small enough so a
certain standard construction of analytic discs filling a wedge with edge
$V$ for generic submanifolds applies.  Let us explain the details.

Let $\sD^n$ be the set of analytic discs in $\C^n$ of class $C^{1,\alpha}$.
Let $\langle V \rangle$ be the set of constant discs attached to $V$.
There exists a neighborhood $\sU$ of $\langle V \rangle$
in $\sD^n$ such that the set $\sA(V,\sU)$
of analytic discs in $\sU$ that are attached to $V$ is
a Banach submanifold of $\sD^n$ modeled on $\sD^k \times \R^d$,
where $k$ is the CR dimension of $V$ and $d$ is the real codimension of $V$.
See Theorem 6.5.4 in \cite{BER:book}.
Let $\sA_p(V,\sU)$ be the set of $\varphi \in \sA(V,\sU)$
such that
$\varphi(1)=p$.
It can be proved that $\sA_p(V,\sU)$ is also a Banach
submanifold modeled on $\sD^n_0$, that is, those discs for
which $\varphi(1)=0$; see Proposition 6.5.3 in \cite{BER:book}.
In particular, assuming we take $\sU$ to be connected, the discs in
$\sA_p(V,\sU)$ can be contracted in $\sA_p(V,\sU)$ to the constant disc
$p$.
Picking a path from the constant disc $p$ to any disc so that
the distance to the constant disc in the underlying Banach space is increasing
gives that the uniform norm of the discs in the family is increasing,
and hence the distance of their centers to $p$ is also increasing.
Therefore, if it can be shown that the discs in $\sA_p(V,\sU)$
cover a neighborhood
of $p$ in $\C^n$, then $V$ satisfies the contracting disc hull condition.

Theorem 8.2.8 in \cite{BER:book} immediately implies the above lemma.
It says, among other things,
that there exists an $\epsilon > 0$ such that if there exists a disc
$\varphi \in \sA_p(V,\sU)$ within $\epsilon$ of the constant disc $p$
such that
$\frac{d}{d t}|_{t=0} \varphi(e^{it}) \in T_p^c V$, then
for all $\delta > 0$ the discs in $\sA_p(V,\sU)$ that are
$\delta$-close to $\varphi$ fill a neighborhood of $p \in \C^n$.

We now obtain a corollary of our main theorem for submanifolds satisfying the
hypotheses of the lemma.

\begin{cor}
Let $V\subset \C^n$ be a connected, bounded, generic real-analytic submanifold, let $p \in V$,
and assume that $V$ satisfies the conditions of Lemma~\ref{BER:lemma} at $p$.
If $f \colon V \to V$ is a real-analytic CR mapping such that $f(p) = p$
and $Df(p)$ is the identity, then $f$ is the identity.
\end{cor}

\begin{proof}
The lemma says that $V$ satisfies the contracting disc hull condition.
Since $V$ is generic, the span of $T_p V$ is $T_p \C^n$.
So the conditions of Theorem~\ref{thm:funcs} are satisfied, and we obtain
the conclusion.
\end{proof}

%%%%%%%%%%%%%%%%%%%%%%%%%%%%%%%%%%%%%%%%%%%%%%%%%%%%%%%%%%%%%%%%%%%%%%%%%%%%

\section{The theorem for infinitesimal automorphisms}

Let $V \subset \C^n$ be a local real-analytic subvariety, that is,
a closed subset of some open set $U$ that locally near each point
is the common zero set of a family of real-analytic functions.
Denote by $V_{reg}$
the set of regular points, that is, the set of points of $V$ near which $V$
is a real-analytic submanifold.
Let $Z$ be a holomorphic vector field
such that $(\Re Z)_q \in T_q V$ for all $q \in V_{reg}$.  If $V$ is a CR manifold,
then $X = \Re Z$ is called an \emph{infinitesimal CR automorphism}.  This definition
makes sense for any real subvariety.

If we consider the flow of $X$ on the regular part of $V$, we find that it
coincides with the holomorphic flow of $Z$ for real times.  So we can just
consider the holomorphic flow of either $X$ or $Z$ and simply restrict to real time $t$.

If a flow exists for some time in some open set of which $V$ is a closed subvariety,
then the flow of $X$ takes $V$ back to $V$.
That this holds at regular points is standard, but then the flow takes the
regular points to the regular points at least for a short time. By
real-analyticity, all regular points must stay in $V$.
That the flow takes all points back to $V$ then follows by continuity as the
set of regular points of $V$ is dense in $V$.

When we talk about $X$ we mean that $X$ defined in a whole
neighborhood of a point in $\C^n$, and we are thinking of it really as $\Re Z$ for $Z$ defined
on an open set.  Nothing is lost, as the extension to some neighborhood is
unique under the hypotheses of the theorem.

\begin{thm} \label{thm:vf}
Suppose $V \subset \C^n$ is a bounded local real-analytic subvariety
that satisfies the contracting disc hull condition at $p \in V$.
Suppose that
$X$ is an infinitesimal CR automorphism on $V$ whose flow on some
neighborhood of $V$ exists for all
positive time and such that the coefficients of $X$ are all $O(2)$ at $p$.
Then $X \equiv 0$.
\end{thm}

We do not require any condition on the tangent cone
since we are assuming that
the coefficients of $X$,
as functions defined in a whole neighborhood of the origin,
are $O(2)$.

\begin{proof}
The proof follows in the same way as that of Theorem~\ref{thm:funcs}.
We consider the time-one flow $f$ of $Z$. We note that, when restricted
to $V$, the iterates of $f$ are simply the flow of $Z$ restricted to $V$
for positive integer times, and we know that all of these exist.
As germs at $p$, and hence as power series at $p$, the time-$k$ flow is
equal to $k$th iterate of $f$.  Let us call this function $f^k$.
Each one of these functions extends uniquely holomorphically to a function
$F^k$ in some polydisc centered at $p$,
and (as before) when restricted to $V$ the iterates $F^k$ are just $f^k$.
We have unique extension to the polydisc, and we find that the power series
of $F^k$ is simply the composition of the power series of $F$.

Now as the coefficients of $X$ at the origin are $O(2)$, the derivative of
the flow at the origin is the identity.

The rest of the proof follows in exactly the same way as that of
Theorem~\ref{thm:funcs}.
\end{proof}

\begin{example} \label{example:cone}
Asking for $O(1)$ is not sufficient.
Consider the subvariety $V \subset \C^2$ given by
\begin{equation}
\sabs{z}^2-\sabs{w}^2 = 0 , \qquad \sabs{z}^2+\sabs{w}^2 < 1.
\end{equation}
Intersections with vertical and horizontal complex lines give rise to
analytic discs attached to $V$ that fill an entire neighborhood of the origin, and
 clearly they shrink to the origin.  So $V$ satisfies the
contracting disc hull condition at the origin.
The vector fields
\begin{equation}
X =
\Re
\left(
  (i a     z  + b   w )  \frac{\partial}{\partial z}
+ (\bar{b} z  + i c w )  \frac{\partial}{\partial w}
\right)
\end{equation}
for $a,c \in \R$ and $b \in \C$ are infinitesimal CR automorphisms of
$V$ at the origin.  In
fact, those are all the linear vector fields.  For the right choice of
$a,b,c$, the flow exists for all time.

Note that Corollary~\ref{cor:CRautcirc} below implies that these vector
fields are the only CR automorphisms of $V$ that fix the origin and exist for all
time.
\end{example}

%20210311 on pCloud
%\begin{example}
%For example, if $V \subset \C^2$ is
%given by
%\begin{equation}
%\Im w= (\Re w) \sabs{z}^2
%\end{equation}
%then the vector fields
%\begin{equation}
%X =
%i \alpha z \frac{\partial}{\partial z} + \beta w \frac{\partial}{\partial w}
%-i \alpha \bar{z} \frac{\partial}{\partial \bar{z}} + \beta \bar{w}
%\frac{\partial}{\partial \bar{w}}
%\end{equation}
%for $\alpha,\beta \in \R$ are infinitesimal CR automorphisms and for small
%enough $\alpha,\beta$ the flow exists for all time.
%\end{example}
%

%%%%%%%%%%%%%%%%%%%%%%%%%%%%%%%%%%%%%%%%%%%%%%%%%%%%%%%%%%%%%%%%%%%%%%%%%%%%

\section{Automorphisms of circular subvarieties}

A set $V \subset \C^n$ is \emph{circular} if, whenever $p \in V$, we have
$e^{i\theta} p \in V$ for all $\theta \in \R$.  We obtain as an immediate
corollary of the preceding results an analogue for real-analytic subvarieties of the classical result on classification of biholomorphisms between circular domains.

\begin{cor}
Suppose $V_1, V_2 \subset \C^n$ are bounded circular real-analytic
subvarieties such that $V_1$ satisfies the contracting disc hull condition
at $0 \in V_1$ and such that the complex span of $C_0 V_1$ is $T_0 \C^n$.
Suppose $f \colon V_1 \to V_2$ is a bijective map in  $\sO(V_1)$ whose
inverse is in $\sO(V_2)$, and suppose $f(0) = 0$.  Then $f$ is linear.
\end{cor}

\begin{proof} Fix $\theta\in\R$.
The map $g(z) = f^{-1}\bigl(e^{-i\theta}f(e^{i\theta} z)\bigr)$ takes $V_1$ to
$V_1$.  By the chain rule, $Dg(0) = I$.  By Theorem~\ref{thm:funcs},
$f^{-1}\bigl(e^{-i\theta}f(e^{i\theta} z)\bigr) = z$ for $z \in V_1$, or in other words,
%\begin{equation}
$f(e^{i\theta} z) = e^{i\theta}f(z)$.
%\end{equation}
Without loss of generality, suppose that $f$ is the holomorphic extension of $f$ near the
origin.
Because, near the origin, any function in $\sO(V_1)$ extends uniquely as a holomorphic function to a
neighborhood, we find that
$f(e^{i\theta} z) = e^{i\theta}f(z)$ for all $z$ in a neighborhood of $0$ in
$\C^n$.
Write the power series of the extension of $f$ as
$f(z) = \sum_{k=1}^\infty f_k(z)$, where the $f_k$ are
homogeneous polynomials of degree $k$.
Then we have
\begin{equation}
\sum_{k=1}^\infty e^{i\theta} f_k(z)
=
e^{i\theta} \sum_{k=1}^\infty f_k(z)
=
\sum_{k=1}^\infty f_k(e^{i\theta} z)
=
\sum_{k=1}^\infty e^{ik\theta}f_k(z) .
\end{equation}
These equations hold for all $\theta$, so uniqueness of the power series dictates that
$f_k \equiv 0$ if $k \not= 1$, and we are done.
\end{proof}

Similarly, we can classify infinitesimal automorphisms.

\begin{cor} \label{cor:CRautcirc}
Suppose $V \subset \C^n$ is a bounded circular local real-analytic subvariety
that satisfies the contracting disc hull condition at $0 \in V$.
Suppose that
$X$ is an infinitesimal CR automorphism on $V$ whose flow on some
neighborhood of $V$ exists for all
positive time and such that $X$ vanishes at $0$.
Then $X$ is linear, and the eigenvalues of the matrix representing $X$
have nonpositive real parts.
\end{cor}

\begin{proof}
We follow the same procedure applied to the flow of $Z$ (where $X = \Re Z$).
The procedure implies that for all positive times $t$, the flow
$F^t(z)$ is linear in $z$.  Then consider the equation of the flow,
$\frac{d}{dt} F^t(z) = Z\bigl(F^t(z)\bigr)$.  The fact that
$F^t$ is linear and $F^t$ is invertible implies that $Z$ is also linear,
thus given by a matrix.  Then so is $X$.  As the equation of flow of $X$ has
a stable critical point, it is classical that the eigenvalues of the matrix that gives $X$
must have nonpositive real parts.
\end{proof}

For an example of using this corollary, note that it applies to
the subvariety $V$ in Example~\ref{example:cone}.

%%%%%%%%%%%%%%%%%%%%%%%%%%%%%%%%%%%%%%%%%%%%%%%%%%%%%%%%%%%%%%%%%%%%%%%%%%%%

\section{On the erratum in the published version}
\label{section:exampleerratum}

In the original (published) version of this paper,
the condition about the distance on the path
in the definition of $\widehat{K}_{CD,p}$
was not present in the definition and hence was also missing from the
definition of the contracting disc hull condition.
That is, the definition was given as
\emph{$\widehat{K}_{CD}$ as the set of points $z \in \C^n$
such that for every $\epsilon > 0$ there exists a continuous
family of analytic discs $\varphi_t \colon \overline{\D} \to \C^n$ ($t \in [0,1]$)
such that the boundary  $\varphi_t(\partial \D)$ lies within an
$\epsilon$-neighborhood of $K$ for all $t$, $\varphi_1(0) = z$, and
the entire disc $\varphi_0(\overline{\D})$ lies within an
$\epsilon$-neighborhood of $K$.}
This definition is not sufficient to guarantee extension of holomorphic
functions to a fixed neighborhood.  Fortunately, adding the condition on the
distance fixes this issue and all the applications examples above do in fact
satisfy the new stricter definition.  Let us give an example of why
the original version of Lemma 2.1 (which used $\widehat{K}_{CD}$
rather than $\widehat{K}_{CD,p}$) as
given above fails.

\begin{example}
Let $K \subset \C^2$ be the set
$K = K_1 \cup K_2 \cup K_3$ where
\begin{align*}
K_1 &= \left\{ (z,w) : \sabs{z-1}^2+\sabs{w}^2 = 1/4 , \Im z \leq 1/4 \right\} , \\
K_2 &= \left\{ (z,0) : \sabs{z-1} \geq 1/4 , \sabs{z} = 1 \right\} , \\
K_3 &= \left\{ (z,0) : \sabs{z-1} \leq 1/4 , \sabs{z} = 1, \Im z \geq 0 \right\} .
\end{align*}
See the diagram in Figure~\ref{figure:diagram}.  Let $p=(1,0)$.
Since the $K_1$ part of the set $K$ is a sphere
with a cap cut off (a ``fishbowl'') we can fill a neighborhood of $p$ with discs attached to $K$
by simply using complex lines.
By moving these lines, we can continuously shrink these discs to the point where the semicircular
path leading to $p$ leaves the fishbowl.
We obtain a family of discs contracting to $p$ that cover a neighborhood
of $p$ with their interiors.
However, the paths created by these families will never be
of increasing distance as required by the new definition. 
Write $g(z,w)$ for a branch of $\log z$ in $\C^2$ that is holomorphic in a neighborhood of $K$ and
that equals $0$
at $p$.  The function
\[
f(z,w) = \frac{1}{i \epsilon + g(z,w)}
\]
is holomorphic in a neighborhood of $K$ yet has a singularity arbitrarily close to $p$
for small $\epsilon > 0$.  Clearly functions in $\sO(K)$ do not extend to a fixed
neighborhood of $p$.
\begin{figure}
\includegraphics{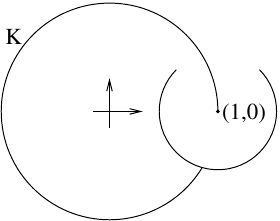}
\caption{Diagram of the set $K$ as seen in the $z$-plane with the point
$(1,0)$ marked.  Note that the smaller circle represents an entire sphere
minus a cap (a ``fishbowl''), while the larger arc is simply an arc in the $z$-plane.%
\label{figure:diagram}}
\end{figure}

We can thus see that the issue with the original definition of the hull is that if a path
created by a family of discs is allowed to leave and come back towards $p$,
 the extension may then be on a different leaf of the envelope of holomorphy.
\end{example}

We need to iterate the construction above
to find a set $V$ which satisfies the uncorrected contracting disc hull condition at $p$
(that is using $\widehat{K}_{CD}$ rather than the new
$\widehat{K}_{CD,p}$)
and fails to admit holomorphic extension to a fixed neighborhood of $p$.
Therefore, Lemma 2.2 and thus also Theorem 2.4 using the uncorrected
definition fail.

\begin{example}
Let $K$ be as before.
Let $h(z,w)$ be a branch of the function $\sqrt{z}$ that is holomorphic
in a neighborhood of $K$ such that $h(p)=1$.  The map
$\phi(z,w) = (h(z,w),w)$ is biholomorphic in a neighborhood of $K$,
and it 
``pulls the $p$ out of the fishbowl''
(it moves the fishbowl and keeps $p$ at $p$) when we look at $\phi(K)$.
Let $\tilde{K}_0 = K$.  Suppose $\tilde{K}_\ell$ is defined,
let $A(\tilde{K}_\ell) =
\delta (\tilde{K}_\ell -p) + p$
be a scaling and translation of $\tilde{K}_\ell$ that preserves $p$.
Here $\delta > 0$ is small enough to
shrink $\tilde{K}_\ell$ so that
$A(\tilde{K}_{\ell})$ fits into the set where the inverse of $\phi$
is still biholomorphic and so that the set
$\phi^{-1}(A(\tilde{K}_{\ell}))$ intersects $K$ only at $p$
and lies inside the set covered by the discs attached to $K$ (attached to
the fishbowl) as constructed before.
Then define
\[
\tilde{K}_{\ell+1} =
K \cup \phi^{-1}(A(\tilde{K}_{\ell})) .
\]
Clearly, $\tilde{K}_{\ell+1} \supset \tilde{K}_{\ell}$, and the set
$V = \bigcup_{\ell=0}^\infty \tilde{K}_\ell$ is compact, locally closed,
and locally connected.
The set $\phi(V)$ contains $A(\tilde{K}_0)$,
which contains the scaled and
moved copy $K$, that is, $A(K)$.
Inverting $A$, we find that $A^{-1}(\phi(V))$ contains a copy of $K$
itself, and moreover $\phi$ is biholomorphic on a neighborhood of
$A^{-1}(\phi(V))$.  We can now repeat the construction of $f$, picking
$\epsilon$ not too small so that the singularity misses
$V$ (enough to ensure that it misses $\phi^{-1}(A(K))$) but so that
the point $(-i\epsilon,0)$ still lies within the set
covered by the attached discs to $K$ as before.
Reversing
the biholomorphisms we applied, we find a function that does not extend to
the neighborhood given by analytic discs covering the
(image of the) fishbowl
in $\phi^{-1}(A(K))$ (we can push forward those disc families).

Any neighborhood of $p$ in $V$ contains some image of $K$ after finitely
many of these given affine moves and mappings by $\phi$.
So, iterating the procedure, we find a function that is holomorphic on a
neighborhood of $V$ but does not extend holomorphically to a fixed
neighborhood of $p$ (as we can iterate sufficiently many times to put the
image of the singularity arbitrarily close to $p$).
This all despite any neighborhood of $p$ in $V$ having
a contracting disc hull in the uncorrected sense that contains a
neighborhood of $p$ in $\C^2$ (the uncorrected sense is preserved
by biholomorphisms, so it is sufficient to construct the discs on $K$).
In particular, the uncorrected
Lemma 2.2, and hence also the uncorrected Theorem 2.4, do not hold.
\end{example}

%%%%%%%%%%%%%%%%%%%%%%%%%%%%%%%%%%%%%%%%%%%%%%%%%%%%%%%%%%%%%%%%%%%%%%%%%%%%

%\section{Data availability statement}
%
%Data sharing not applicable to this article as no datasets were generated or analyzed during the current study.

%%%%%%%%%%%%%%%%%%%%%%%%%%%%%%%%%%%%%%%%%%%%%%%%%%%%%%%%%%%%%%%%%%%%%%%%%%%%

\def\MR#1{\relax\ifhmode\unskip\spacefactor3000 \space\fi%
  \href{http://mathscinet.ams.org/mathscinet-getitem?mr=#1}{MR#1}}


\begin{bibdiv}
\begin{biblist}

\bib{BER:98}{article}{
   author={Baouendi, M. S.},
   author={Ebenfelt, P.},
   author={Rothschild, Linda Preiss},
   title={CR automorphisms of real analytic manifolds in complex space},
   journal={Comm. Anal. Geom.},
   volume={6},
   date={1998},
   number={2},
   pages={291--315},
   issn={1019-8385},
   review={\MR{1651418}},
%   doi={10.4310/CAG.1998.v6.n2.a3},
}

\bib{BER:book}{book}{
  author={Baouendi, M. Salah},
  author={Ebenfelt, Peter},
  author={Rothschild, Linda Preiss},
  title={Real submanifolds in complex space and their mappings},
  series={Princeton Mathematical Series},
  volume={47},
  publisher={Princeton University Press, Princeton, NJ},
  date={1999},
  pages={xii+404},
  isbn={0-691-00498-6},
  %review={\MR{1668103 (2000b:32066)}},
  review={\MR{1668103}},
}

\bib{BER:00}{article}{
   author={Baouendi, M. S.},
   author={Ebenfelt, P.},
   author={Preiss Rothschild, Linda},
   title={Local geometric properties of real submanifolds in complex space},
   journal={Bull. Amer. Math. Soc. (N.S.)},
   volume={37},
   date={2000},
   number={3},
   pages={309--336},
   issn={0273-0979},
   review={\MR{1754643}},
%   doi={10.1090/S0273-0979-00-00863-6},
}

\bib{Beloslapka:88}{article}{
   author={Beloshapka, V. K.},
   title={Finite-dimensionality of the group of automorphisms of a real
   analytic surface},
   language={Russian},
   journal={Izv. Akad. Nauk SSSR Ser. Mat.},
   volume={52},
   date={1988},
   number={2},
   pages={437--442, 448},
   issn={0373-2436},
   translation={
      journal={Math. USSR-Izv.},
      volume={32},
      date={1989},
      number={2},
      pages={443--448},
      issn={0025-5726},
   },
   review={\MR{941685}},
%   doi={10.1070/IM1989v032n02ABEH000775},
}

\bib{ChernMoser:74}{article}{
   author={Chern, S. S.},
   author={Moser, J. K.},
   title={Real hypersurfaces in complex manifolds},
   journal={Acta Math.},
   volume={133},
   date={1974},
   pages={219--271},
   issn={0001-5962},
   review={\MR{425155}},
%   doi={10.1007/BF02392146},
}

\bib{FLF:19}{article}{
   author={Bertrand, Florian},
   author={Blanc-Centi, L\'{e}a},
   author={Meylan, Francine},
   title={Stationary discs and finite jet determination for non-degenerate
   generic real submanifolds},
   journal={Adv. Math.},
   volume={343},
   date={2019},
   pages={910--934},
   issn={0001-8708},
   review={\MR{3891986}},
%   doi={10.1016/j.aim.2018.12.005},
}

\bib{ELZ:03}{article}{
   author={Ebenfelt, P.},
   author={Lamel, B.},
   author={Zaitsev, D.},
   title={Finite jet determination of local analytic CR automorphisms and
   their parametrization by 2-jets in the finite type case},
   journal={Geom. Funct. Anal.},
   volume={13},
   date={2003},
   number={3},
   pages={546--573},
   issn={1016-443X},
   review={\MR{1995799}},
%   doi={10.1007/s00039-003-0422-y},
}

\bib{Joricke:09}{article}{
   author={J\"{o}ricke, Burglind},
   title={Envelopes of holomorphy and holomorphic discs},
   journal={Invent.\ Math.},
   volume={178},
   date={2009},
   number={1},
   pages={73--118},
   issn={0020-9910},
   review={\MR{2534093}},
%   doi={10.1007/s00222-009-0194-6},
}

\bib{KaupUpmeier}{article}{
   author={Kaup, Wilhelm},
   author={Upmeier, Harald},
   title={An infinitesimal version of Cartan's uniqueness theorem},
   journal={Manuscripta Math.},
   volume={22},
   date={1977},
   number={4},
   pages={381--401},
   issn={0025-2611},
   review={\MR{466651}},
%   doi={10.1007/BF01168224},
}


\bib{LamelMir:07}{article}{
   author={Lamel, Bernhard},
   author={Mir, Nordine},
   title={Finite jet determination of CR mappings},
   journal={Adv. Math.},
   volume={216},
   date={2007},
   number={1},
   pages={153--177},
   issn={0001-8708},
   review={\MR{2353253}},
%   doi={10.1016/j.aim.2007.05.007},
}


\bib{Lebl:scv}{misc}{
   author={Lebl, Ji\v{r}\'i},
   title={Tasty Bits of Several Complex Variables},
   edition={4th edition},
   note={\url{https://www.jirka.org/scv/}}
}

%\bib{LNR:CRimages}{article}{
%%   author={Lebl, Ji\v{r}\'i},
 %  author={Alan Noell},
 %  author={Sivaguru Ravisankar},
 %  title={On CR singular CR images},
 %  note={accepted to Internat.\ J.\ Math., preprint \url{https://arxiv.org/abs/2012.01820}}
%}

\bib{LNR:CRimages}{article}{
   author={Lebl, Ji\v{r}\'{\i}},
   author={Noell, Alan},
   author={Ravisankar, Sivaguru},
   title={On CR singular CR images},
   journal={Internat. J. Math.},
   volume={32},
   date={2021},
   number={13},
   pages={Paper No. 2150090, 26 pp.},
   issn={0129-167X},
   review={\MR{4361992}},
%   doi={10.1142/S0129167X21500907},
}


\bib{Lee:06}{article}{
   author={Lee, Kang-Hyurk},
   title={Almost complex manifolds and Cartan's uniqueness theorem},
   journal={Trans. Amer. Math. Soc.},
   volume={358},
   date={2006},
   number={5},
   pages={2057--2069},
   issn={0002-9947},
   review={\MR{2197447}},
%   doi={10.1090/S0002-9947-05-03973-5},
}


\bib{Poletsky}{article}{
   author={Poletsky, Evgeny A.},
   title={Holomorphic currents},
   journal={Indiana Univ. Math. J.},
   volume={42},
   date={1993},
   number={1},
   pages={85--144},
   issn={0022-2518},
   review={\MR{1218708}},
%   doi={10.1512/iumj.1993.42.42006},
}

\bib{Porten}{article}{
   author={Porten, Egmont},
   title={Polynomial hulls and analytic discs},
   journal={Proc.\ Amer.\ Math.\ Soc.},
   volume={145},
   date={2017},
   number={10},
   pages={4443--4448},
   issn={0002-9939},
   review={\MR{3690627}},
%   doi={10.1090/proc/13596},
}

%\bib{Shabat:book}{book}{
%   author={Shabat, B.~V.},
%   title={Introduction to complex analysis. Part II},
%   series={Translations of Mathematical Monographs},
%   volume={110},
%   note={Functions of several variables;
%   Translated from the third (1985) Russian edition by J. S. Joel},
%   publisher={American Mathematical Society},
%   place={Providence, RI},
%   date={1992},
%   pages={x+371},
%   isbn={0-8218-4611-6},
%   review={\MR{1192135}},
%}

\bib{Tumanov:20}{article}{
   author={Tumanov, Alexander},
   title={Stationary discs and finite jet determination for CR mappings in
   higher codimension},
   journal={Adv. Math.},
   volume={371},
   date={2020},
   pages={107254, 11},
   issn={0001-8708},
   review={\MR{4116100}},
%   doi={10.1016/j.aim.2020.107254},
}

%\bib{Whitney:book}{book}{
%  author={Whitney, Hassler},
%  title={Complex analytic varieties},
%  publisher={Addison-Wesley Publishing Co., Reading, Mass.-London-Don
%  Mills, Ont.},
%  date={1972},
%  pages={xii+399},
%  review={\MR{0387634}},
%}

\bib{Zaitsev:02}{article}{
   author={Zaitsev, Dmitri},
   title={Unique determination of local CR-maps by their jets: a survey},
   note={Harmonic analysis on complex homogeneous domains and Lie groups
   (Rome, 2001)},
   journal={Atti Accad. Naz. Lincei Cl. Sci. Fis. Mat. Natur. Rend. Lincei
   (9) Mat. Appl.},
   volume={13},
   date={2002},
   number={3-4},
   pages={295--305},
   issn={1120-6330},
   review={\MR{1984108}},
}

\end{biblist}
\end{bibdiv}
\end{document}